\documentclass{article}
\usepackage{arxiv}
\usepackage {changes}
\usepackage{placeins}
\graphicspath{{images/}}

\title{Coupled Stochastic Multigrid Reconstruction for Blind Ptychography}

\hypersetup{
hidelinks,
pdftitle={Coupled Stochastic Multigrid Reconstruction for Blind Ptychography},
pdfsubject={math, optics},
pdfauthor={Borong Zhang, Junjing Deng, Yi Jiang, Zichao Wendy Di},
pdfkeywords={phase retrieval,  multigrid optimization,  inverse problems,  ptychography, coherent diffraction imaging},
}

\author{ \href{https://orcid.org/0009-0001-9020-9152}{\includegraphics[scale=0.06]{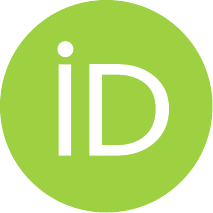}\hspace{1mm}Borong Zhang} \\
	School of Mathematics\\
	Georgia Institute of Technology\\
	Atlanta, GA 30332 \\
	\And
	\href{https://orcid.org/0000-0003-4472-3385}{\includegraphics[scale=0.06]{orcid.pdf}\hspace{1mm}Junjing Deng} \\
	Advanced Photon Source\\
	Argonne National Laboratory\\
	Lemont, IL 60439\\
    \And
	\href{https://orcid.org/0000-0002-2807-1324}{\includegraphics[scale=0.06]{orcid.pdf}\hspace{1mm} Yi Jiang} \\
	Advanced Photon Source\\
	Argonne National Laboratory\\
	Lemont, IL 60439\\
    \And
	\href{https://orcid.org/0000-0002-4131-9363}{\includegraphics[scale=0.06]{orcid.pdf}\hspace{1mm}Zichao Wendy Di}\thanks{Corresponding author. Code and results: \url{https://github.com/borongzhang/blind_magpie_ptychography}.} \\
	Mathematics and Computer Science \& Advanced Photon Source\\
	Argonne National Laboratory\\
	Lemont, IL 60439\\
}

\begin{document}
\maketitle

\begin{abstract}
Blind ptychography enables quantitative imaging when both the specimen and illumination probe are unknown, but modern data sets call for reconstruction methods that can assimilate many diffraction patterns while resolving structure across spatial scales. We introduce a coupled stochastic multigrid framework that addresses these requirements within a unified optimization construction. Rather than appending multigrid correction to an existing stochastic algorithm, the framework begins with a common surrogate for the local object and probe updates and connects these updates to minibatch consensus and coarse-grid correction through a unified residual analysis. This analysis establishes local control of the geometric-mean coupling and characterizes its extension across a grid hierarchy. The fine-grid construction yields the geometric-mean regularized Ptychographical Iterative Engine (GM-rPIE), while the inclusion of multilevel corrections yields the geometric-mean Multilevel--Adaptive--Guided Ptychographic Iterative Engine (GM-MAGPIE). Across nine synthetic configurations spanning scan overlap and Poisson noise, GM-MAGPIE achieves the lowest final object error at the prescribed budgets, while either GM-rPIE or GM-MAGPIE achieves the lowest probe error. Reconstructions of experimental test-pattern and integrated-circuit data further demonstrate the framework's applicability to measured X-ray data.
\end{abstract}


\section{Introduction}
Ptychography reconstructs complex-valued object images from diffraction intensities recorded as a coherent probe scans overlapping regions. It supports applications in materials science~\cite{materials_science_2}, biological imaging~\cite{biology_1,biology_2}, and integrated-circuit metrology~\cite{integrated_circuit,integrated_circuit_1}; see the survey by Rodenburg and Maiden~\cite{Rodenburg2019}. In practice, the probe is often unknown and must be recovered together with the object. This \emph{blind} reconstruction couples the two fields through their product, the exit wave, while the detector records only its Fourier intensity.

Blind reconstruction must organize information along two distinct axes. Along the data axis, stochastic or minibatch methods assimilate subsets of diffraction patterns at each update. Along the spatial axis, multigrid methods transfer information between fine and coarse representations so that corrections can act over different length scales. Combining these ideas requires more than appending a coarse-grid step to a stochastic algorithm. The object and probe share the same exit waves, while local updates from overlapping scan positions must remain consistent with common global fields. This motivates a framework in which stochastic processing and multilevel correction arise from a shared optimization structure.

The extended and regularized variants of the Ptychographical Iterative Engine (PIE)~\cite{PIE}, ePIE~\cite{ePIE} and rPIE~\cite{rPIE}, update both the object and the probe. A stochastic minibatch extension of rPIE~\cite{Tripathi2024} combines information across scan positions through separate subproblems for common object and probe estimates. Multigrid methods provide a complementary strategy for propagating corrections across spatial scales~\cite{multigrid_book}; coarse representations have also been used in ptychographic phase retrieval~\cite{MGOPT_ptycho}. More recently, the Multilevel--Adaptive--Guided Ptychographic Iterative Engine (MAGPIE)~\cite{magpie} combined a majorizing surrogate with stochastic multigrid object updates in the non-blind setting. These developments leave open how to construct a blind stochastic multigrid method in which the local, minibatch, and multilevel updates remain analytically connected.

In this work, we introduce a coupled stochastic multigrid framework for blind ptychography. A common frozen surrogate defines separate local object and probe updates on the fine grid, while a hierarchy of coarser grids extends the object update across spatial scales. A residual analysis then connects these local fine- and multilevel updates to the synthesis of common fields across each minibatch. Within this construction, entrywise complex geometric means coordinate simultaneous changes to the object and probe, and intensity-weighted consensus rules combine information across scan positions. Using only the fine grid gives GM-rPIE, whereas the multilevel construction gives the geometric-mean Multilevel--Adaptive--Guided Ptychographic Iterative Engine (GM-MAGPIE).

The analysis provides guarantees at each stage of this construction. For exact rPIE block updates, the geometric-mean step contracts the local exit-wave residual and does not increase the frozen local surrogate. For multigrid object updates, a perturbation result describes the coordinated residual near the revised exit wave. A minibatch residual bound then separates the local residuals from disagreements introduced when the object and probe updates are synthesized into common fields. Together, these results provide a principled path from a local surrogate model to stochastic multilevel reconstruction.

We compare GM-rPIE and GM-MAGPIE with rPIE and least-squares maximum-likelihood (LSQML) reconstruction~\cite{LSQML}. At the prescribed budgets, GM-MAGPIE produces the smallest final object error in all nine tested synthetic configurations spanning scan overlap and Poisson noise, while one of the two geometric-mean methods produces the smallest probe error. We further apply the methods to test-pattern and integrated-circuit measurements from the Advanced Photon Source. Taken together, the analysis and experiments support treating stochastic data processing and multilevel correction as a coupled algorithmic design rather than as independent reconstruction components.

\section{Blind Ptychography and a Common Surrogate}
\label{sec:prelim}

\subsection{Forward model and amplitude objective}
\label{sec:blind_ptychography}

Let $\bm O\in\mathbb C^{n^2}$ and $\bm P\in\mathbb C^{m^2}$ denote the vectorized $n\times n$ object and $m\times m$ probe, respectively. At each of $N$ scan positions, the probe illuminates a local region of the object, and the detector records the resulting diffraction intensity.

At scan position $k$, let $\mathcal E_k$ extract the corresponding $m\times m$ object patch and let $S_k$ apply the corresponding, possibly subpixel, shift to the probe. The local object and probe fields are then
\begin{equation}
    \bm O_k=\mathcal E_k\bm O,\qquad
    \bm P_k=S_k\bm P.
    \label{eq:local_fields}
\end{equation}
The transpose $\mathcal E_k^{\top}$ places a local patch back into its corresponding region of the full object grid, with zeros elsewhere. In our implementation, $S_k$ pads the probe with a one-pixel border of replicated boundary values, applies the subpixel shift using the Fourier shift theorem, and crops the shifted field to the original probe dimensions.

The product $\bm P_k\odot\bm O_k$ is the exit wave leaving the illuminated object region, where $\odot$ denotes entrywise multiplication. Its Fourier intensity determines the diffraction pattern. We model the measured data as
\begin{equation}
    \bm d_k
    =\left|\mathcal F\bigl(
       \bm P_k^*\odot\bm O_k^*
       \bigr)\right|^2+\bm\epsilon_k,
    \qquad k=1,\ldots,N.
    \label{eq:blind_forward_model}
\end{equation}
Here $\bm d_k\in\mathbb R^{m^2}$ contains the measured intensities at the detector pixels, and $\bm\epsilon_k\in\mathbb R^{m^2}$ represents measurement noise. The superscript $*$ denotes the unknown ground-truth fields, with $\bm O_k^*=\mathcal E_k\bm O^*$ and $\bm P_k^*=S_k\bm P^*$. The operator $\mathcal F$ is the unitary two-dimensional discrete Fourier transform, and the magnitude and square are evaluated entrywise. Figure~\ref{fig:setup} illustrates the blind ptychography setup.

For a regular scanning grid with step size $s$ and probe width $m$, the overlap ratio between adjacent scanning windows along each scan direction is $1-s/m$. For example, $50\%$ overlap corresponds to a scan step equal to half the probe width. Higher overlap increases redundancy, but also acquisition time and dose at fixed flux and exposure per pattern~\cite{Rodenburg2019}.

\begin{figure}[htbp]
    \centering
    \includegraphics[width=0.62\linewidth]{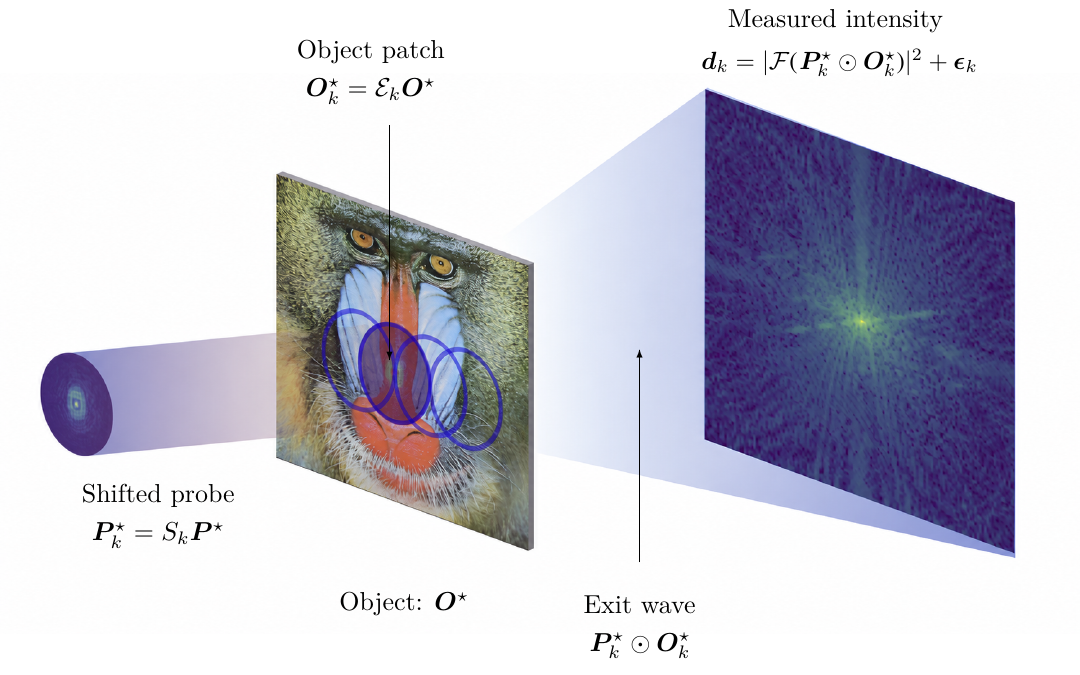}
    \caption{Blind ptychography setup. The shifted probe illuminates an object patch, and the detector records the Fourier intensity of the resulting exit wave.}
    \label{fig:setup}
\end{figure}

We reconstruct the two fields by minimizing the amplitude misfit
\begin{equation*}
    \Phi(\bm P,\bm O)\defeq\frac12\sum_{k=1}^N
    \left\|\sqrt{\bm d_k}-\left|\mathcal F(\bm P_k\odot\bm O_k)\right|\right\|_2^2.
\end{equation*}
The transformation $(\bm P,\bm O)\mapsto(c\bm P,c^{-1}\bm O)$ leaves the exit waves unchanged for every nonzero complex $c$. This scaling ambiguity requires alignment and normalization when reconstructions are compared with a reference.

To derive the object and probe updates, we rewrite the amplitude misfit as a discrepancy between exit waves in the object plane. For this purpose, we introduce the revised exit wave, which combines the measured Fourier amplitudes with the phases of the current Fourier prediction.

\begin{definition}[Revised exit wave]
\label{def:REW}
For local fields $(\bm P_k,\bm O_k)$, define
\begin{equation}
    \mathcal R_k(\bm P_k,\bm O_k)\defeq
    \mathcal F^{-1}\!\left[\sqrt{\bm d_k}\odot
    e^{i\theta(\mathcal F(\bm P_k\odot\bm O_k))}\right],
    \label{eq:revised_exit_wave}
\end{equation}
where $\theta$ is the entrywise principal argument, with phase zero selected at zero Fourier coefficients.
\end{definition}
Parseval's identity gives the equivalent exit-wave form
\begin{equation}
    \Phi(\bm P,\bm O)=\sum_{k=1}^N\Phi_k(\bm P_k,\bm O_k),\qquad
    \Phi_k\defeq\frac12\left\|\mathcal R_k(\bm P_k,\bm O_k)-\bm P_k\odot\bm O_k\right\|_2^2.
    \label{eq:exit_misfit}
\end{equation}

\subsection{Frozen surrogate and rPIE block updates}
\label{sec:surrogate}
\label{sec:pie}

PIE-type algorithms, including rPIE, process diffraction patterns one scanning region at a time, typically in a randomly shuffled order~\cite{rPIE}. This random ordering gives a stochastic sequence of local updates. We now derive the object and probe updates used at each step of this sequence.

At a given scan update during epoch $j$, let $(\bm P^{(j)},\bm O^{(j)})$ denote the current working object--probe pair. This pair changes as scanning regions are processed, whereas $j$ advances only after a complete pass through the data. Suppressing the within-epoch scan index, we define the local fields, revised exit waves, and residuals by
\begin{equation}
\begin{aligned}
    \bm P_k^{(j)}&=S_k\bm P^{(j)},&
    \bm O_k^{(j)}&=\mathcal E_k\bm O^{(j)},\\
    \bm R_k^{(j)}
    &=\mathcal R_k(\bm P_k^{(j)},\bm O_k^{(j)}),&
    \bm r_k^{(j)}
    &=\bm R_k^{(j)}-\bm P_k^{(j)}\odot\bm O_k^{(j)}.
\end{aligned}
    \label{eq:current_local_fields}
\end{equation}
We use \emph{exit-wave residual} for the revised exit wave minus the predicted product, \emph{reconstruction error} for a difference from ground truth, and \emph{misfit} for a scalar measure of a residual.

Holding the revised exit waves $\bm R_k^{(j)}$ fixed at the current working pair defines the surrogate
\begin{equation}
\begin{aligned}
    \widetilde\Phi(\bm P,\bm O;j)
    &=\sum_k\widetilde\Phi_k(\bm P_k,\bm O_k;j),\\
    \widetilde\Phi_k(\bm P_k,\bm O_k;j)
    &\defeq\frac12\left\|
        \bm R_k^{(j)}-\bm P_k\odot\bm O_k
    \right\|_2^2.
\end{aligned}
    \label{eq:quadratic_surrogate}
\end{equation}
The revised exit waves are recomputed from the current working pair and held fixed while the local object and probe updates are formed. The resulting surrogate follows the majorization--minimization principle~\cite{mm_alg_Lange} as it bounds the amplitude objective from above and coincides with it at the current pair. Consequently, a step that decreases this bound from its current value also decreases the corresponding amplitude misfit. Appendix~\ref{supp:derivation} establishes these properties.

With the revised exit waves fixed, the surrogate is quadratic in each field when the other is held fixed. We therefore construct the local object and probe updates by solving separate regularized block subproblems.
\begin{subequations}
\label{eq:rpie_block_subproblems}
\begin{align}
    \widehat{\bm O}_k^{(j)}
    &\defeq\operatorname*{argmin}_{\bm O_k}
    \left\{
    \widetilde\Phi_k(\bm P_k^{(j)},\bm O_k;j)
    +\tfrac12\left\langle
        \bm u_{\bm O},|\bm O_k-\bm O_k^{(j)}|^2
    \right\rangle
    \right\},
    \label{eq:rpie_object_subproblem}\\
    \widehat{\bm P}_k^{(j)}
    &\defeq\operatorname*{argmin}_{\bm P_k}
    \left\{
    \widetilde\Phi_k(\bm P_k,\bm O_k^{(j)};j)
    +\tfrac12\left\langle
        \bm u_{\bm P},|\bm P_k-\bm P_k^{(j)}|^2
    \right\rangle
    \right\},
    \label{eq:rpie_probe_subproblem}
\end{align}
\end{subequations}
where the nonnegative regularizers $\bm u_{\bm O}$ and $\bm u_{\bm P}$ penalize changes from the current object and probe fields, respectively. Each regularizer is evaluated at the current value of the opposite field and held fixed during minimization.

All quotients involving vectors are interpreted \emph{entrywise}. Assuming entrywise positive denominators, the exact block minimizers are
\begin{subequations}
\label{eq:rpie_block_minimizers}
\begin{align}
    \widehat{\bm O}_k^{(j)}
    &=\bm O_k^{(j)}
    +\frac{\overline{\bm P_k^{(j)}}}
          {|\bm P_k^{(j)}|^2+\bm u_{\bm O}}
    \odot\bm r_k^{(j)},
    \label{eq:rpie_object_minimizer}\\
    \widehat{\bm P}_k^{(j)}
    &=\bm P_k^{(j)}
    +\frac{\overline{\bm O_k^{(j)}}}
          {|\bm O_k^{(j)}|^2+\bm u_{\bm P}}
    \odot\bm r_k^{(j)}.
    \label{eq:rpie_probe_minimizer}
\end{align}
\end{subequations}
The usual rPIE choices are~\cite{rPIE}
\begin{equation}
    \bm u_{\bm O}(\bm P_k)=\alpha_O\bigl(
        \|\bm P_k\|_\infty^2\bm1-|\bm P_k|^2
    \bigr),\qquad
    \bm u_{\bm P}(\bm O_k)=\alpha_P\bigl(
        \|\bm O_k\|_\infty^2\bm1-|\bm O_k|^2
    \bigr),
    \label{eq:rpie_regularizers}
\end{equation}
with $\alpha_O,\alpha_P>0$. Each single-scan denominator therefore has the form $(1-\alpha)|\cdot|^2+\alpha\|\cdot\|_\infty^2\bm1$. In the minibatch implementation of the geometric-mean methods, $\max_{\ell\in B}\|\bm P_\ell\|_\infty^2$ replaces the per-scan probe maximum in the object regularizer. The probe regularizer uses the object maximum for each patch.

With unit step sizes, rPIE applies these block minimizers directly~\cite{rPIE}. Applying both simultaneously, however, need not decrease the frozen surrogate in Eq.~\eqref{eq:quadratic_surrogate}, because both factors in the object--probe product change. The next section coordinates these changes locally, and Section~\ref{sec:minibatch_gm_magpie} then combines the local fields across each minibatch.

\section{Local Geometric-Mean and Multigrid Updates}
\label{sec:parallel_block_updates}

\subsection{GM-rPIE geometric-mean updates}

The stochastic construction begins from the separate rPIE block updates at each scan position in a minibatch. Before combining information across scan positions, we coordinate the fine-grid object and probe changes within each scanning region. Because both updates use the same frozen working pair and residual, we combine each updated value with its current field through an entrywise geometric mean. This rule defines the fine-grid variant GM-rPIE.

\begin{definition}[Complex geometric mean]
\label{def:anchored_geometric_mean}
For $a,b\in\mathbb C$, define
\begin{equation}
    \operatorname{gm}(a,b)\defeq
    \begin{cases}
        a\,\operatorname{sqrt}_{\mathrm{pr}}(b/a),&a\ne0,\\
        0,&a=0,
    \end{cases}
    \label{eq:anchored_geometric_mean}
\end{equation}
where $\operatorname{sqrt}_{\mathrm{pr}}$ denotes the principal square root, obtained by choosing the phase of its input in $(-\pi,\pi]$ and halving that phase. The definition applies entrywise to vectors.
\end{definition}
For nonzero inputs, this operation takes the geometric mean of the magnitudes and halves the principal phase difference from the current value $a$ to the updated value $b$. Thus the square-root branch is selected relative to the current field. GM-rPIE forms
\begin{equation}
    \widetilde{\bm P}_k^{(j)}
    =\operatorname{gm}(\bm P_k^{(j)},\widehat{\bm P}_k^{(j)}),\qquad
    \widetilde{\bm O}_k^{(j)}
    =\operatorname{gm}(\bm O_k^{(j)},\widehat{\bm O}_k^{(j)}).
    \label{eq:geometric_mean_update}
\end{equation}
The following bound explains why this combination is useful. At each entry, the combined update has a residual magnitude no greater than the larger residual produced by the two separate block updates, with all residuals measured relative to the same fixed revised exit wave.
\begin{proposition}[Entrywise exit-wave residual contraction]
\label{prop:gm_rpie_local_descent}
Fix a scan update at position $k$ and epoch counter $j$, with the local fields and revised exit wave in Eq.~\eqref{eq:current_local_fields} held fixed. Suppose that the regularizers $\bm u_{\bm O}$ and $\bm u_{\bm P}$ are entrywise nonnegative and that the denominators in the exact rPIE block updates~\eqref{eq:rpie_block_minimizers} are strictly positive. Let $\widetilde{\bm O}_k^{(j)}$ and $\widetilde{\bm P}_k^{(j)}$ be the coordinated fields obtained from these block updates through Eq.~\eqref{eq:geometric_mean_update}, using the anchored square-root branch in Definition~\ref{def:anchored_geometric_mean}. For this coordinated update, define
\begin{equation}
    \bm\kappa_{\bm O,k}^{(j)}
    =\frac{\bm u_{\bm O}}{|\bm P_k^{(j)}|^2+\bm u_{\bm O}},\qquad
    \bm\kappa_{\bm P,k}^{(j)}
    =\frac{\bm u_{\bm P}}{|\bm O_k^{(j)}|^2+\bm u_{\bm P}},\qquad
    \bm\gamma_k^{(j)}
    =\max\{\bm\kappa_{\bm O,k}^{(j)},
             \bm\kappa_{\bm P,k}^{(j)}\}.
    \label{eq:gm_contraction_factors}
\end{equation}
The maximum is taken entrywise, and the assumptions imply $0\leq\bm\gamma_k^{(j)}\leq1$ entrywise. The coordinated exit-wave residual then satisfies
\begin{equation}
    \left|\bm R_k^{(j)}-\widetilde{\bm P}_k^{(j)}\odot
    \widetilde{\bm O}_k^{(j)}\right|
    \leq\bm\gamma_k^{(j)}\odot|\bm r_k^{(j)}|.
    \label{eq:entrywise_gm_contraction}
\end{equation}
Consequently, the frozen local surrogate in Eq.~\eqref{eq:quadratic_surrogate} satisfies
\begin{equation}
    \widetilde\Phi_k
    (\widetilde{\bm P}_k^{(j)},\widetilde{\bm O}_k^{(j)};j)
    \leq\tfrac12\left\langle
    (\bm\gamma_k^{(j)})^2,|\bm r_k^{(j)}|^2\right\rangle\leq\widetilde\Phi_k
    (\bm P_k^{(j)},\bm O_k^{(j)};j).
    \label{eq:vector_gm_contraction}
\end{equation}
The final inequality is strict if there exists an entry $s$ for which $r_{k,s}^{(j)}\ne0$, $P_{k,s}^{(j)}\ne0$, and $O_{k,s}^{(j)}\ne0$.
\end{proposition}
The proof is in Appendix~\ref{app:geometric_mean_descent}. We next extend the object update across multiple spatial scales while retaining the same coordination mechanism and fine-grid probe update.

\subsection{GM-MAGPIE multilevel object update}
\label{sec:gm_magpie_development}
\label{sec:multigrid_object_motivation}
\label{sec:gm_magpie}

To incorporate coarse-scale object information, we replace the rPIE object update with the recursive multigrid map $\mathcal M_L(\bm O_k,\bm P_k,\bm R_k,\bm u_{\bm O})$. The probe update remains on the fine grid, and the two fields are coordinated using the geometric-mean rule in Eq.~\eqref{eq:geometric_mean_update}. This extension defines GM-MAGPIE.

Starting from the local object, probe, revised exit wave, and object regularizer, $\mathcal M_L$ transfers these quantities through an $L$-level grid hierarchy. It applies a regularized object update on the coarsest grid and then returns to the fine grid by prolonging each coarse correction and applying another regularized update at each finer level. For $L=1$, the map reduces to the rPIE object update. An $L$-level hierarchy requires $m$ to be divisible by $2^{L-1}$. Appendix~\ref{supp:multigrid} gives the full construction, while Appendix~\ref{app:magpie_transfer} provides the transfer and coarse-correction bounds inherited from MAGPIE.

We coordinate the resulting multigrid object update $\widehat{\bm O}_{k,\mathrm{MG}}^{(j)}$ with the unchanged rPIE probe update using Eq.~\eqref{eq:geometric_mean_update}. In contrast to the fine-grid setting of Proposition~\ref{prop:gm_rpie_local_descent}, the exit-wave residual produced by the multigrid object update need not be collinear in the complex plane with the residual produced by the probe update. The exact contraction bound therefore does not follow directly. The next proposition instead gives a local first-order characterization of the coordinated residual.

\begin{proposition}[First-order residual under multigrid coordination]
\label{prop:magpie_geometric_mean_perturbation}
\label{lem:magpie_geometric_mean_perturbation}
Fix a scan update at position $k$ and epoch counter $j$, and hold the local fields and revised exit wave fixed. Let $\widehat{\bm O}_{k,\mathrm{MG}}^{(j)}$ be the multigrid object update, let $\widehat{\bm P}_k^{(j)}$ be the rPIE probe update, and let $\widetilde{\bm O}_k^{(j)}$ and $\widetilde{\bm P}_k^{(j)}$ be their coordinated values from Eq.~\eqref{eq:geometric_mean_update}. Define the residuals from the separate object and probe updates by
\[
\begin{aligned}
    \bm e_{\bm O,k}^{(j)}
    &\defeq \bm R_k^{(j)}-\bm P_k^{(j)}
       \odot\widehat{\bm O}_{k,\mathrm{MG}}^{(j)},\\
    \bm e_{\bm P,k}^{(j)}
    &\defeq \bm R_k^{(j)}-\widehat{\bm P}_k^{(j)}
       \odot\bm O_k^{(j)}.
\end{aligned}
\]
For each entry $s$ such that $P_{k,s}^{(j)}O_{k,s}^{(j)}\ne0$ and $R_{k,s}^{(j)}\ne0$, assume that the anchored geometric means satisfy
\[
    \widetilde P_{k,s}^{(j)}\widetilde O_{k,s}^{(j)}
    \longrightarrow R_{k,s}^{(j)}
    \quad\text{as}\quad
    \left(e_{\bm O,k,s}^{(j)},e_{\bm P,k,s}^{(j)}\right)
    \longrightarrow(0,0).
\]
Then, as the two separate-update residuals approach zero with $\bm R_k^{(j)}$ fixed, the combined exit-wave residual satisfies
\begin{equation}
    \bm R_k^{(j)}-\widetilde{\bm P}_k^{(j)}
    \odot\widetilde{\bm O}_k^{(j)}=\tfrac12\left(
       \bm e_{\bm O,k}^{(j)}+\bm e_{\bm P,k}^{(j)}
       \right)+\mathcal O\!\left(
       \frac{\left(
       |\bm e_{\bm O,k}^{(j)}|
       +|\bm e_{\bm P,k}^{(j)}|
       \right)^2}{|\bm R_k^{(j)}|}
       \right),
\label{eq:local_multigrid_gm_expansion}
\end{equation}
where the limit and the $\mathcal O$ estimate apply only at entries satisfying these assumptions.
\end{proposition}

To first order, the geometric-mean step averages the residuals from the separate object and probe updates. Near the revised exit wave, an object update whose residual magnitude is much smaller than that of the probe update therefore leaves approximately half of the probe-update residual. Appendix~\ref{app:magpie_perturbation} provides the proof.

\section{Minibatch Synthesis}
\label{sec:minibatch_gm_magpie}

Section~\ref{sec:parallel_block_updates} produces a coordinated object--probe pair at each scan position, whereas the reconstruction requires one object and probe shared by all positions. Synthesizing inconsistent overlapping patches or shifted probe estimates can increase the exit-wave residual. We bound this contribution and use the bound to construct an intensity-weighted synthesis rule.

\subsection{Local updates and synthesis disagreements}
\label{sec:minibatch_local_gm_magpie}
\label{sec:minibatch_linearized_product_consistency}

For the current minibatch $B$, compute all local updates from the same frozen working pair $(\bm P^{(j)},\bm O^{(j)})$. The local fields, revised exit waves, and regularizers are recomputed at the start of each minibatch and held fixed while its updates are formed. The probe update is given by Eq.~\eqref{eq:rpie_probe_minimizer}, while the object update is
\begin{equation}
    \widehat{\bm O}_k^{(j)}
    =\mathcal M_L\!\left(
       \bm O_k^{(j)},\bm P_k^{(j)},\bm R_k^{(j)},
       \bm u_{\bm O,k}^{(j)}
       \right).
    \label{eq:minibatch_magpie_endpoint}
\end{equation}
Here $\mathcal M_L$ is the recursive $L$-level object update defined in Section~\ref{sec:gm_magpie_development}, and $\bm u_{\bm O,k}^{(j)}$ is the object regularizer evaluated for the current scan and minibatch. Taking $L=1$ gives the rPIE object update and hence GM-rPIE. Using additional levels gives GM-MAGPIE. Equation~\eqref{eq:geometric_mean_update} then forms the coordinated local fields $(\widetilde{\bm P}_k^{(j)},\widetilde{\bm O}_k^{(j)})$ to be synthesized.

For candidate global fields $(\bm P,\bm O)$, write their disagreements with these coordinated local fields as
\[
    \Delta\bm P_k
    =S_k\bm P-\widetilde{\bm P}_k^{(j)},\qquad
    \Delta\bm O_k
    =\mathcal E_k\bm O-\widetilde{\bm O}_k^{(j)}.
\]
The resulting exit-wave residual has the exact decomposition
\begin{equation}
\begin{aligned}
    \bm R_k^{(j)}-(S_k\bm P)\odot(\mathcal E_k\bm O)
    &=
    \bm R_k^{(j)}-\widetilde{\bm P}_k^{(j)}
    \odot\widetilde{\bm O}_k^{(j)}\\
    &-\widetilde{\bm O}_k^{(j)}\odot\Delta\bm P_k
     -\widetilde{\bm P}_k^{(j)}\odot\Delta\bm O_k\\
    &-\Delta\bm P_k\odot\Delta\bm O_k.
\end{aligned}
    \label{eq:minibatch_product_expansion}
\end{equation}
Applying the triangle inequality to the concatenated minibatch
residuals gives
\begin{equation}
\begin{aligned}
    \left(
    \tfrac12\sum_{k\in B}
    \|\bm R_k^{(j)}-(S_k\bm P)\odot(\mathcal E_k\bm O)\|_2^2
    \right)^{1/2}
    &\leq
    \left(
    \tfrac12\sum_{k\in B}
    \|\bm R_k^{(j)}-\widetilde{\bm P}_k^{(j)}
      \odot\widetilde{\bm O}_k^{(j)}\|_2^2
    \right)^{1/2}\\
    &+
    \left(
    \tfrac12\sum_{k\in B}
    \bigl\||\widetilde{\bm O}_k^{(j)}|
      \odot\Delta\bm P_k\bigr\|_2^2
    \right)^{1/2}+
    \left(
    \tfrac12\sum_{k\in B}
    \bigl\||\widetilde{\bm P}_k^{(j)}|
      \odot\Delta\bm O_k\bigr\|_2^2
    \right)^{1/2}\\
    &+
    \left(
    \tfrac12\sum_{k\in B}
    \|\Delta\bm P_k\odot\Delta\bm O_k\|_2^2
    \right)^{1/2}.
\end{aligned}
    \label{eq:minibatch_separable_bound}
\end{equation}
By majorization, the same right-hand side also bounds $\bigl(\sum_{k\in B}\Phi_k(S_k\bm P,\mathcal E_k\bm O)\bigr)^{1/2}$.

The four terms have distinct roles. The first measures the exit-wave residuals remaining after the local geometric-mean updates. Section~\ref{sec:parallel_block_updates} bounds these residuals for GM-rPIE and analyzes them near the revised exit wave for GM-MAGPIE. The second and third terms measure the disagreements introduced by probe and object synthesis, respectively. The last term is second order in their joint disagreements. To choose the synthesis rules, we form separate intensity-weighted consensus problems from the two first-order terms. The second-order interaction remains in the bound but is omitted from these synthesis problems.

\subsection{Intensity-weighted synthesis}
\label{sec:minibatch_weighted_synthesis}
\label{sec:minibatch_subpixel_probe_shift}

The synthesis terms in Eq.~\eqref{eq:minibatch_separable_bound} weight object disagreements by the probe magnitude and probe disagreements by the object magnitude. Minimizing the object disagreement gives an exact intensity-weighted average. For the probe, a common-grid approximation gives an analogous average.

In the following problems, the coordinated local fields are fixed and the unindexed fields $(\bm P,\bm O)$ denote candidate common fields. Both synthesis updates are evaluated from the same working pair before either field is overwritten.

The object consensus problem is
\begin{equation}
    \min_{\bm O}\;
    \tfrac12\sum_{k\in B}
    \left\|
    |\widetilde{\bm P}_k^{(j)}|
    \odot(\mathcal E_k\bm O-\widetilde{\bm O}_k^{(j)})
    \right\|_2^2.
    \label{eq:minibatch_object_consensus}
\end{equation}
Because $\mathcal E_k$ extracts an object patch, the problem separates over object pixels. Its exact solution is the following intensity-weighted average of the overlapping local object fields.
\begin{equation}
    \bm O^{(j)}\gets
    \frac{
    \displaystyle\sum_{k\in B}\mathcal E_k^{\top}
    \left[
    |\widetilde{\bm P}_k^{(j)}|^2
    \odot\widetilde{\bm O}_k^{(j)}
    \right]}
    {\displaystyle\sum_{k\in B}\mathcal E_k^{\top}
    |\widetilde{\bm P}_k^{(j)}|^2}.
    \label{eq:minibatch_object_synthesis}
\end{equation}
Pixels with zero total weight retain their current values.

The corresponding probe problem is
\begin{equation}
    \min_{\bm P}\;
    \tfrac12\sum_{k\in B}
    \left\|
    |\widetilde{\bm O}_k^{(j)}|
    \odot(S_k\bm P-\widetilde{\bm P}_k^{(j)})
    \right\|_2^2.
    \label{eq:minibatch_shifted_probe_consensus}
\end{equation}
We move the local fields to the common probe grid using the approximation
\begin{equation}
    \left\|
        |\widetilde{\bm O}_k^{(j)}|
        \odot(S_k\bm P-\widetilde{\bm P}_k^{(j)})
        \right\|_2^2\approx
        \left\|
        |S_k^{\mathrm H}\widetilde{\bm O}_k^{(j)}|
        \odot(\bm P-S_k^{\mathrm H}\widetilde{\bm P}_k^{(j)})
        \right\|_2^2.
    \label{eq:minibatch_probe_shift_approximation}
\end{equation}
This change of coordinates is exact for ideal translations that preserve norms and pointwise products, but is approximate for discrete fractional interpolation and padded boundaries.

The resulting probe-consensus problem is
\begin{equation}
    \min_{\bm P}\;
    \tfrac12\sum_{k\in B}
    \left\|
    |S_k^{\mathrm H}\widetilde{\bm O}_k^{(j)}|
    \odot(\bm P-S_k^{\mathrm H}\widetilde{\bm P}_k^{(j)})
    \right\|_2^2.
    \label{eq:minibatch_probe_common_grid}
\end{equation}
It separates over probe pixels. Where the total weight is positive, its explicit solution is
\begin{equation}
    \bm P^{(j)}\gets
    \frac{
    \displaystyle\sum_{k\in B}
    |S_k^{\mathrm H}\widetilde{\bm O}_k^{(j)}|^2
    \odot\bigl(S_k^{\mathrm H}\widetilde{\bm P}_k^{(j)}\bigr)}
    {\displaystyle\sum_{k\in B}
    |S_k^{\mathrm H}\widetilde{\bm O}_k^{(j)}|^2}.
    \label{eq:minibatch_shifted_probe_synthesis}
\end{equation} 

In the implementation, we form a weighted average of the adjoint-shifted local probe
corrections and add the result to the current probe. The weights are
computed by interpolating the complex object updates with the opposite
displacement before taking their squared magnitudes. These steps agree
with Eq.~\eqref{eq:minibatch_shifted_probe_synthesis} for ideal
translations, whereas padded discrete interpolation makes this an approximate
description. At zero total weight, the approximate probe-consensus problem
leaves the probe unconstrained. Where the implementation's weights have
zero sum, it adds the ordinary mean of the adjoint-shifted local corrections.

The object and probe updates are applied together after each minibatch, and the fields after the final minibatch initialize the next epoch. Algorithm~\ref{alg:gm_methods} summarizes the procedure for fixed scan positions. The regularization settings, position refinement, and epoch-end amplitude normalization used in the experiments are described in Section~\ref{sec:experiments}.

\begin{algorithm}[htbp]
\caption{GM-rPIE and GM-MAGPIE}
\label{alg:gm_methods}
\small
\begin{algorithmic}[1]
\Require Initial fields
$(\bm P^{(0)},\bm O^{(0)})$, diffraction data and scan positions, regularization parameters $\alpha_O,\alpha_P$, batch size $b$, epoch budget $j_{\max}$, and number of object levels $L$.
\State Use $L=1$ for GM-rPIE or $L>1$ valid levels for GM-MAGPIE.
\For{$j=0,\ldots,j_{\max}-1$}
    \State Randomly permute the scan positions and divide them into minibatches of size at most $b$.
    \For{each minibatch $B$}
        \State Freeze $(\bm P^{(j)},\bm O^{(j)})$ and compute the local fields, revised exit waves, and regularizers.
        \State Compute the object and probe updates for all $k\in B$ using Eqs.~\eqref{eq:minibatch_magpie_endpoint} and \eqref{eq:rpie_probe_minimizer}.
        \State Coordinate the local object and probe fields using Eq.~\eqref{eq:geometric_mean_update}.
        \State Evaluate both synthesis updates from the frozen pair, then overwrite $(\bm P^{(j)},\bm O^{(j)})$ using Eqs.~\eqref{eq:minibatch_object_synthesis} and \eqref{eq:minibatch_shifted_probe_synthesis}.
    \EndFor
    \State Set $(\bm P^{(j+1)},\bm O^{(j+1)})\gets (\bm P^{(j)},\bm O^{(j)})$ to carry the updated fields into the next epoch.
\EndFor
\end{algorithmic}
\end{algorithm}

\section{Numerical Examples}
\label{sec:experiments}

We compared rPIE~\cite{rPIE}, the least-squares maximum-likelihood (LSQML) method~\cite{LSQML}, GM-rPIE, and GM-MAGPIE using the updated Pty-Chi implementation~\cite{ptychi}.

\subsection{Comparison protocol}
\label{sec:numerical_metrics}

Minibatch sizes are reported as counts and percentages of the scan positions. Within each configuration, all methods used the same data, initial fields, scan positions, minibatch size, and epoch budget. Each epoch visited every position in shuffled minibatches, with amplitude normalization at epoch end. Synthetic positions remained fixed. GM-rPIE used the finest object grid. GM-MAGPIE used all available dyadic object levels: eight for synthetic data, nine for the test pattern, and ten for the integrated circuit. Both variants used the probe synthesis of Section~\ref{sec:minibatch_gm_magpie}. The PIE step sizes were one. The LSQML object and probe optimal-step scalers were one for synthetic data and $0.5$ for both experimental data sets. LSQML used the Gaussian likelihood.

For the PIE-type methods, the single-mode numerical amplitude replacement on a measured pixel is $z\sqrt{d+\varepsilon}/(|z|+\varepsilon)$, with detector prediction $z$, intensity $d$, and $\varepsilon=10^{-7}$. The majorization and tangency statements in Appendix~\ref{supp:derivation} use the exact amplitude replacement.

Before evaluating synthetic reconstruction errors, we align the reconstructed object with the ground truth by fitting a complex scale and an affine phase ramp on the central $932\times932$ pixels. The object NMSE uses this same region, and the probe NMSE uses the full probe array. We apply the inverse scale and phase-ramp corrections to the probe and then remove its remaining global phase. For the aligned fields, we measure the object and probe errors according to
\begin{equation}
    \operatorname{NMSE}_O\defeq
    \frac{\|\bm O_{\mathrm{al}}-\bm O^*\|_2^2}
         {\|\bm O^*\|_2^2},\qquad
    \operatorname{NMSE}_P\defeq
    \frac{\|\bm P_{\mathrm{al}}-\bm P^*\|_2^2}
         {\|\bm P^*\|_2^2}.
\label{eq:synthetic_truth_metrics}
\end{equation}

\subsection{Synthetic data}
\label{sec:synthetic_results}

The $1024\times1024$ ground-truth object used the Baboon image for its magnitude and the Cameraman image for its phase. The initial object was identically one. The true and initial $128\times128$ probes used defocus settings of $-700$ and $-400~\mu\mathrm m$, respectively (Fig.~\ref{fig:synthetic_setup}).

\begin{figure}[htbp]
    \centering
    \includegraphics[width=\linewidth]{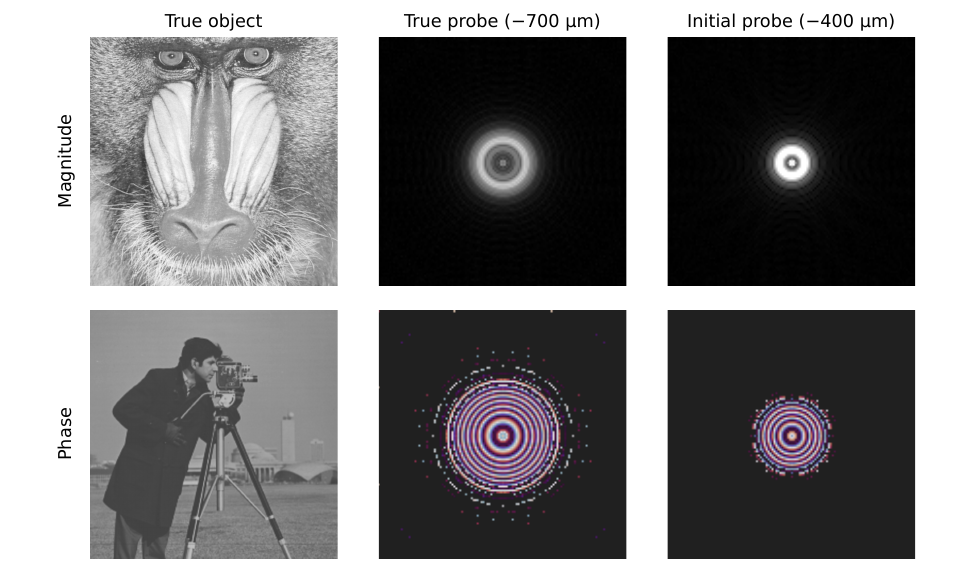}
    \caption{Magnitude and phase of the ground-truth object, true probe, and initial probe. Probe phases use cyclic colors.}
    \label{fig:synthetic_setup}
\end{figure}

The simulated noiseless intensities $\widetilde{\bm d}_k$ use an unnormalized forward discrete Fourier transform. In these intensity units, the noisy data were
\begin{equation}
    d_{k,r}=\eta X_{k,r},\qquad
    X_{k,r}\sim\operatorname{Poisson}(\widetilde d_{k,r}/\eta).
\label{eq:synthetic_poisson_model}
\end{equation}
Thus $\eta$ controls the noise variance, $\eta\widetilde d_{k,r}$, while preserving the noiseless mean.

The nominal overlaps $75\%$, $87.5\%$, and $93.75\%$ correspond to nominal scan steps of 32, 16, and 8 pixels, respectively. The parameter sweeps used 3,000, 800, and 200 epochs and minibatches of 27 scan positions ($3.21\%$, $0.83\%$, and $0.21\%$, respectively).

\subsubsection{Overlap-specific parameter sweeps}

At $\eta=0.1$, each PIE method evaluated nine pairs with $\alpha_O,\alpha_P\in\{0.004,0.02,0.1\}$. LSQML evaluated sixteen pairs with damping in $\{0.04,0.1,0.2,1\}$ and $\sigma\in\{0.3,0.4,0.5,0.6\}$. We selected the pair with the lowest final object NMSE separately for each method and overlap.

GM-MAGPIE attained the lowest final object error at all three overlaps (Fig.~\ref{fig:synthetic_overlap_histories}), improving on GM-rPIE by factors of 1.10, 4.22, and 2.36, respectively. GM-rPIE attained the lowest final probe NMSE at $75\%$ overlap, while GM-MAGPIE attained the lowest at $87.5\%$ and $93.75\%$ overlap (Fig.~\ref{fig:synthetic_overlap_probe_histories}). Figures~\ref{fig:synthetic_overlap_objects} and \ref{fig:synthetic_overlap_probes} show the reconstructed fields at $87.5\%$ overlap. Both geometric-mean variants reproduce the compact probe structure with smaller complex-error regions than rPIE and LSQML.

\begin{figure}[htbp]
    \centering
    \includegraphics[width=\linewidth]{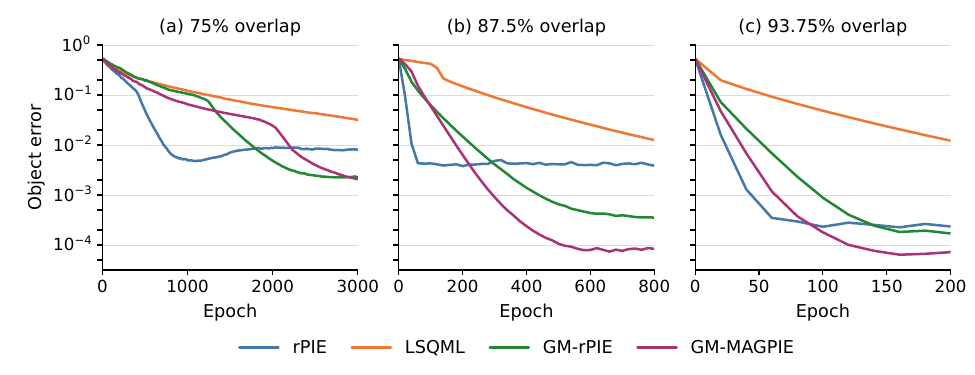}
    \caption{Object-error convergence at $75\%$, $87.5\%$, and $93.75\%$ overlap, with 3,000, 800, and 200 epochs, respectively. Each curve uses its selected pair at $\eta=0.1$. The minibatch size is 27 scan positions ($3.21\%$, $0.83\%$, and $0.21\%$, respectively). Errors were recorded every 20 epochs.}
    \label{fig:synthetic_overlap_histories}
\end{figure}

\begin{figure}[htbp]
    \centering
    \includegraphics[width=\linewidth]{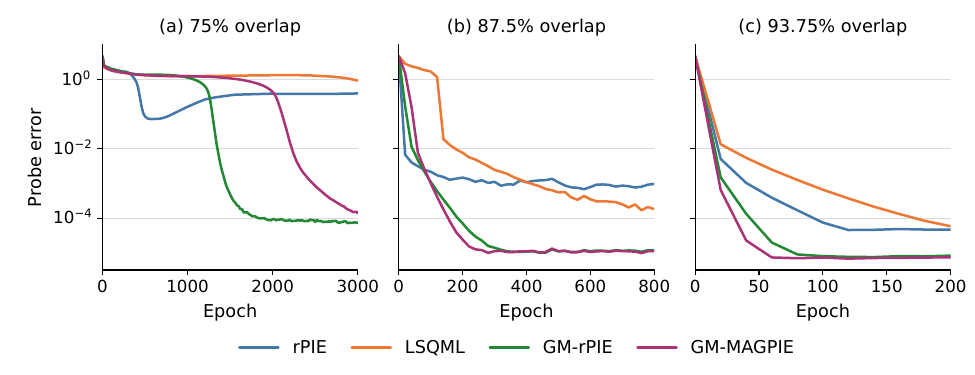}
    \caption{Probe-error convergence corresponding to Fig.~\ref{fig:synthetic_overlap_histories}, using the same object-selected parameter pairs and epoch budgets.}
    \label{fig:synthetic_overlap_probe_histories}
\end{figure}

\begin{figure}[p]
    \centering
    \includegraphics[width=\linewidth]{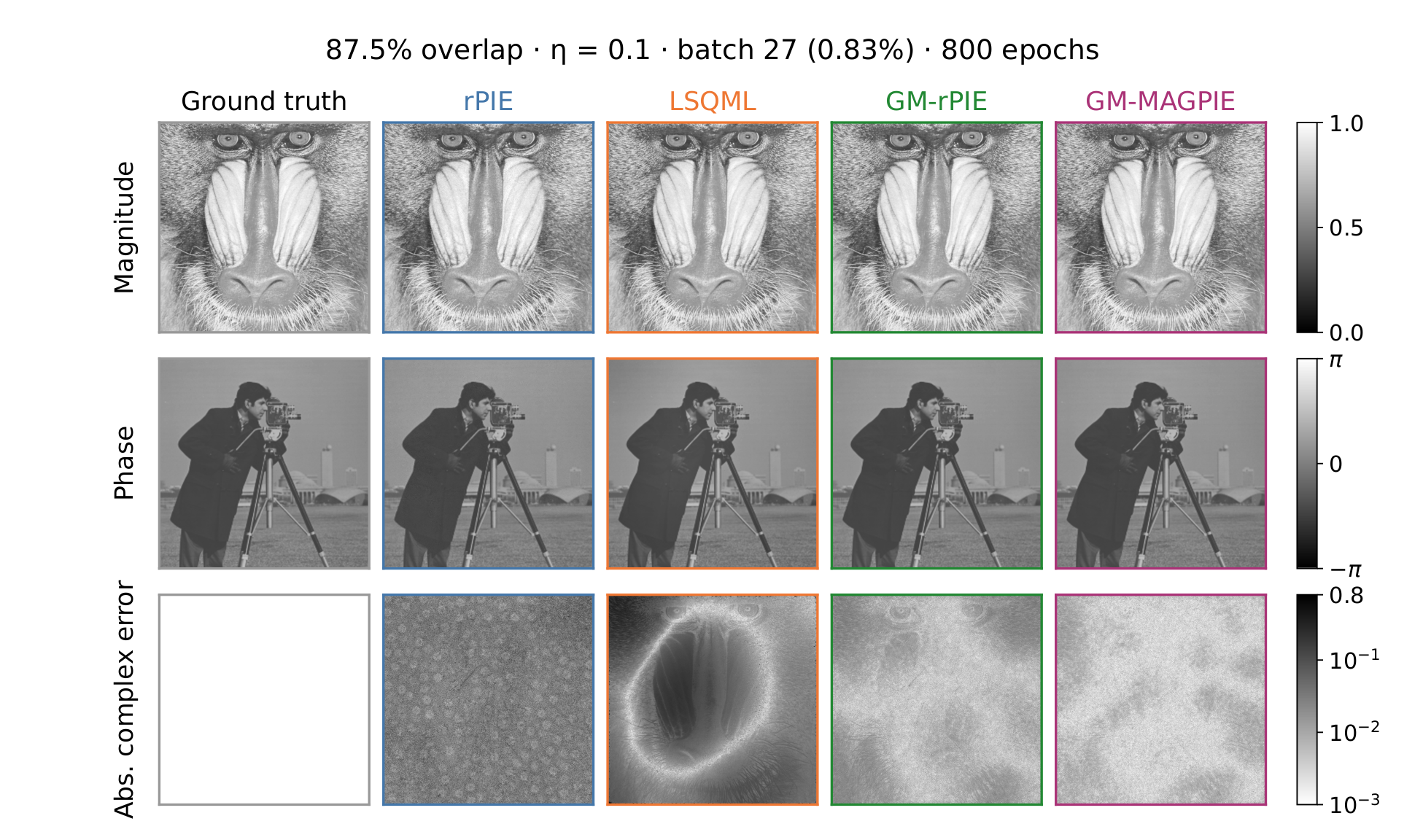}
    \caption{Object reconstructions at $87.5\%$ overlap, $\eta=0.1$, minibatch size 27 ($0.83\%$), and 800 epochs. The rows show magnitude, phase, and absolute complex error $|O_{\mathrm{al}}-O^*|$, while the columns show the truth and four methods. Error images share a logarithmic scale $[10^{-3},0.8]$, with smaller errors displayed in white.}
    \label{fig:synthetic_overlap_objects}
\end{figure}

\begin{figure}[htbp]
    \centering
    \includegraphics[width=\linewidth]{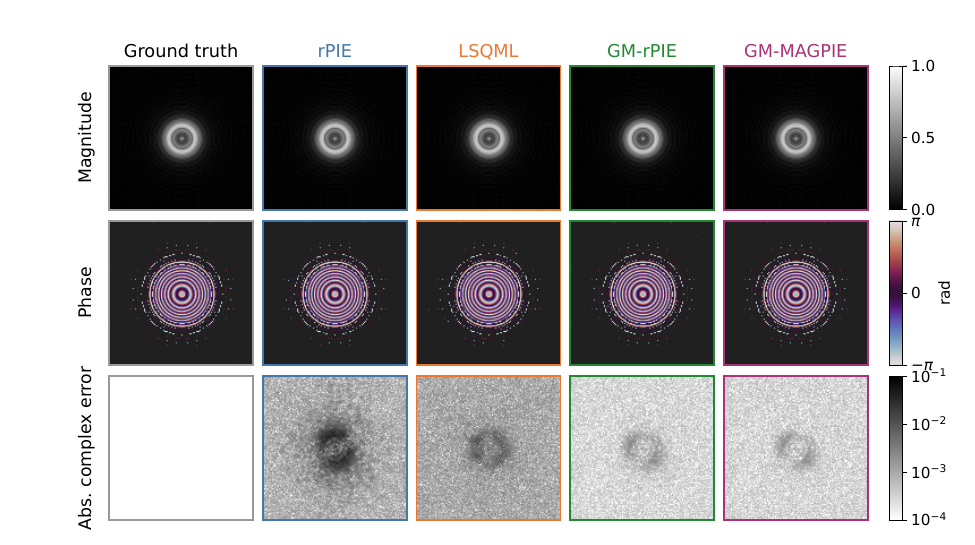}
    \caption{Probe reconstructions corresponding to Fig.~\ref{fig:synthetic_overlap_objects}. The rows show magnitude, phase, and absolute complex error $|P_{\mathrm{al}}-P^*|$, while the columns show the truth and four methods. Each row shares display limits across methods, with a logarithmic scale for the errors. Phase uses cyclic colors, with weak probe phases masked.}
    \label{fig:synthetic_overlap_probes}
\end{figure}

\subsubsection{Noise tests with fixed selected parameters}
\label{sec:noise_results}

At each overlap, we repeated the selected reconstructions at $\eta=0.2$ and $0.4$, retaining the parameters, initialization, scan, and minibatch size. All three noise levels used 3,000, 800, and 200 epochs at $75\%$, $87.5\%$, and $93.75\%$ overlap, respectively.

GM-MAGPIE had the lowest object error in all nine configurations (Fig.~\ref{fig:synthetic_noise_histories}). The lowest final probe NMSE in each configuration was attained by either GM-rPIE or GM-MAGPIE. GM-rPIE had the lowest probe NMSE at $75\%$ overlap for all three noise levels, and GM-MAGPIE had the lowest at the two higher overlaps.

\begin{figure}[htbp]
    \centering
    \includegraphics[width=\linewidth]{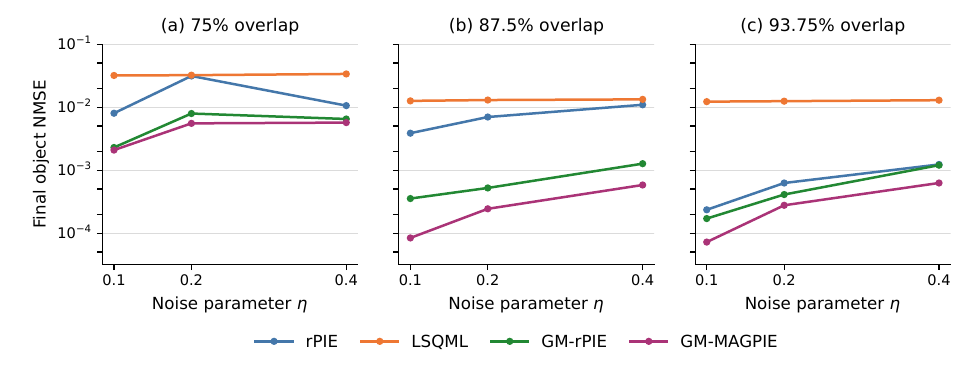}
    \caption{Final object NMSE versus $\eta$ at each overlap, using the parameters selected at $\eta=0.1$. The budgets are 3,000, 800, and 200 epochs. The minibatch size is 27 scan positions ($3.21\%$, $0.83\%$, and $0.21\%$, respectively).}
    \label{fig:synthetic_noise_histories}
\end{figure}

\subsubsection{Minibatch-size tests at 87.5\% overlap}
\label{sec:batch_results}

At $87.5\%$ overlap and $\eta=0.1$, we compared minibatch sizes 27 ($0.83\%$), 108 ($3.32\%$), and 432 ($13.30\%$) for 800 epochs. The same parameter grids were searched independently at each size, selecting by final object NMSE. The 27-position ($0.83\%$) results were reused from the overlap sweep.

GM-MAGPIE had the lowest final object error at each batch size (Fig.~\ref{fig:synthetic_batch_histories}), improving on GM-rPIE by factors of 4.22, 2.30, and 29.48. Larger batches increased the final object error for all four methods under this epoch budget. GM-MAGPIE also attained the lowest final probe error at each minibatch size (Fig.~\ref{fig:synthetic_batch_probe_histories}).

\begin{figure}[htbp]
    \centering
    \includegraphics[width=\linewidth]{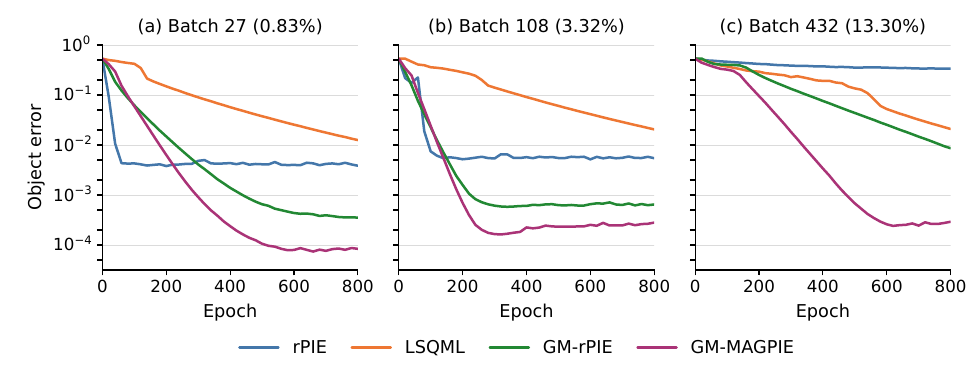}
    \caption{Object-error convergence for minibatch sizes 27 ($0.83\%$), 108 ($3.32\%$), and 432 ($13.30\%$) at $87.5\%$ overlap, $\eta=0.1$, and 800 epochs. Parameters are retuned at each size. The 27-position ($0.83\%$) curves reuse the overlap sweep.}
    \label{fig:synthetic_batch_histories}
\end{figure}

\begin{figure}[htbp]
    \centering
    \includegraphics[width=\linewidth]{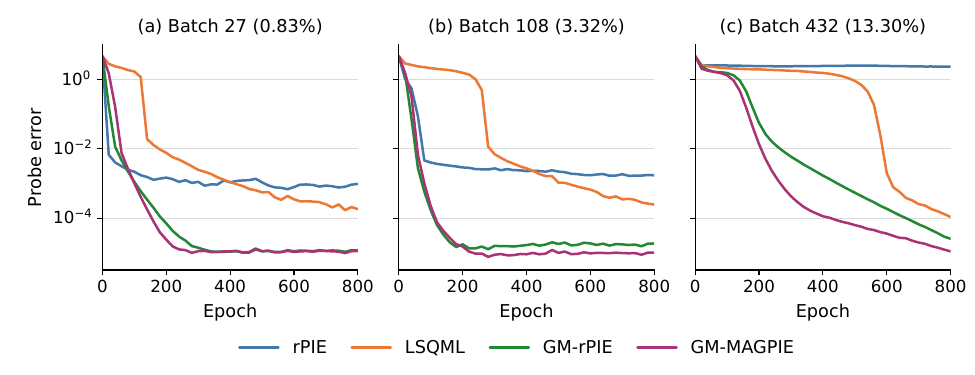}
    \caption{Probe-error convergence corresponding to Fig.~\ref{fig:synthetic_batch_histories}, using the same object-selected parameters at each minibatch size.}
    \label{fig:synthetic_batch_probe_histories}
\end{figure}

\FloatBarrier

\subsection{Experimental data}
\label{sec:experimental_results}

We reconstructed test-pattern and integrated-circuit data acquired with the Velociprobe instrument~\cite{dataset_1} at the Advanced Photon Source (APS), Argonne National Laboratory. Table~\ref{tab:experimental_configurations} gives the data and minibatch sizes. The initial probe was a Fresnel zone-plate simulation normalized to the mean diffraction power, and Table~\ref{tab:experimental_parameters} lists the method parameters. Starting from common coordinates, all methods refined scan positions after every minibatch beginning in the first epoch, using gradient-based correction with Fourier differentiation and a mean-position constraint. This refinement is additional to the fixed-position updates analyzed above.

\begin{table}[htbp]
\centering
\caption{Experimental configurations.}
\label{tab:experimental_configurations}
\small
\begin{tabular}{lrrr}
\toprule
Data set & Positions used & Minibatch size & Epochs \\
\midrule
Test pattern          & 14,640 & 244 (1.67\%) & 100 \\
Integrated circuit, full  & 812   & 29 (3.57\%) & 100 \\
\bottomrule
\end{tabular}
\end{table}

\begin{table}[htbp]
\centering
\caption{Parameters for the experimental data sets. The rPIE and geometric-mean columns give $(\alpha_O,\alpha_P)$. The two experimental LSQML runs use the Gaussian likelihood with $\sigma=0.5$ and preconditioner damping $0.1$. Here $\beta$ is the common object and probe optimal-step scaler.}
\label{tab:experimental_parameters}
\begin{tabular}{lcccc}
\toprule
Data set & rPIE & GM-rPIE & GM-MAGPIE & LSQML \\
\midrule
Test pattern & $(0.2,1)$ & $(0.2,1)$ & $(2,1)$ & $\beta=0.5$ \\
Integrated circuit & $(1,0.5)$ & $(1,0.1)$ & $(3,0.1)$ & $\beta=0.5$ \\
\bottomrule
\end{tabular}
\end{table}

\subsubsection{Test-pattern data}

The test pattern was scanned at beamline 2-ID-D using a 10~keV X-ray beam focused by a Fresnel zone plate with a 50~nm outermost zone width and a $180~\mu\mathrm m$ diameter. A gold Siemens star resolution target with 100~nm innermost spokes was approximately $200~\mu\mathrm m$ from focus, and a Dectris EIGER area detector with a $75~\mu\mathrm m$ pixel size recorded the diffraction patterns 1.95~m downstream. We used 14,640 of the 14,641 supplied $256\times256$ diffraction patterns, and the initial-probe defocus was $-0.3$~mm. All four methods recovered the radial spoke structure in both magnitude and phase (Fig.~\ref{fig:test_pattern_reconstructions}). On the shared display scale, the geometric-mean variants produced sharper spoke boundaries, and GM-MAGPIE gave the clearest visual delineation of the spokes in magnitude. The corresponding probe reconstructions recover the same compact central focus and ring structure (Fig.~\ref{fig:test_pattern_probe_reconstructions}). The geometric-mean variants show less diffuse magnitude background and fewer isolated phase features outside the central probe support. Because neither the experimental object nor the probe is known, these comparisons are qualitative rather than accuracy rankings.

\begin{figure}[tbp]
    \centering
    \includegraphics[width=\linewidth,height=0.82\textheight,keepaspectratio]{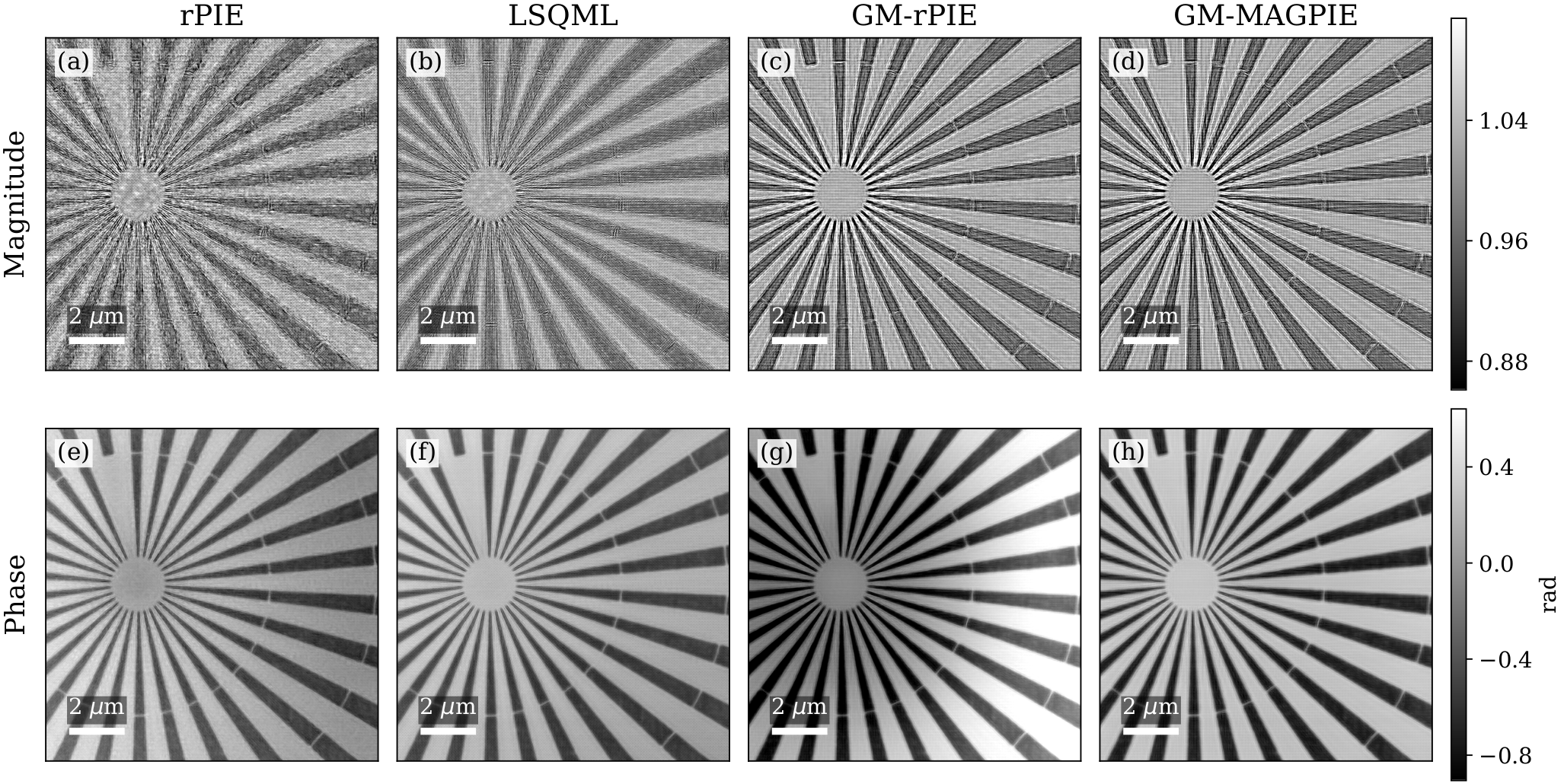}
    \caption{Test-pattern object reconstructions from 14,640 patterns after 100 epochs. The rows show magnitude and phase, while the columns show rPIE, LSQML, GM-rPIE, and GM-MAGPIE. Each row shares display limits across methods. Each panel includes a $2~\mu\mathrm m$ scale bar.}
    \label{fig:test_pattern_reconstructions}
\end{figure}

\begin{figure}[htbp]
    \centering
    \includegraphics[width=\linewidth]{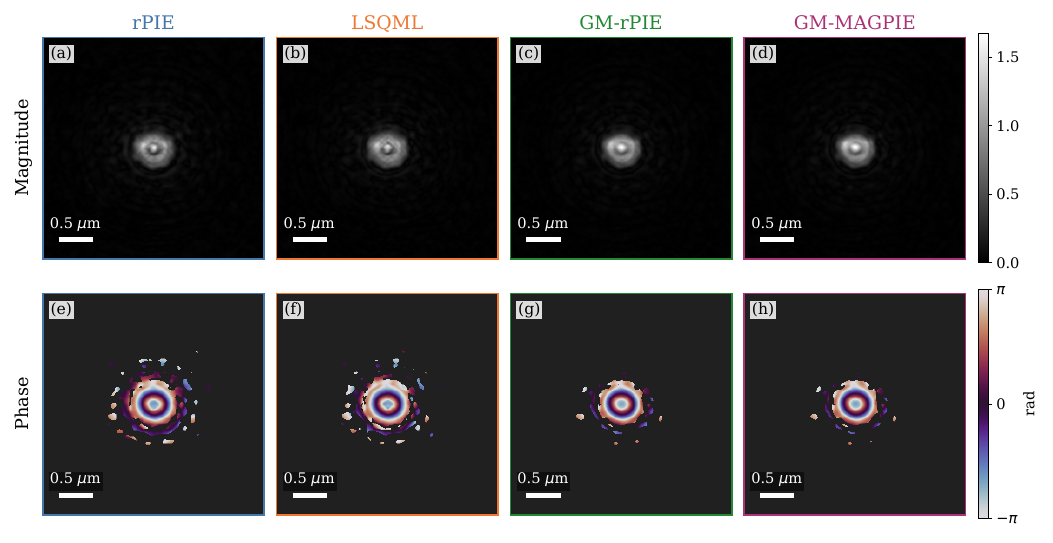}
    \caption{Test-pattern probe reconstructions after 100 epochs, corresponding to Fig.~\ref{fig:test_pattern_reconstructions}. The rows show magnitude and phase, while the columns show rPIE, LSQML, GM-rPIE, and GM-MAGPIE. Each row shares display limits across methods. Each panel includes a $0.5~\mu\mathrm m$ scale bar. Phase uses cyclic colors, with weak probe phases masked.}
    \label{fig:test_pattern_probe_reconstructions}
\end{figure}

\subsubsection{Integrated-circuit data}

A CMOS microelectronics chip~\cite{deng_prb_2017} fabricated with a 65-nm process was scanned using the same experimental setup as the test pattern. Each diffraction pattern collected from the integrated circuit was cropped to a $512\times512$ pixel array. The initial-probe defocus was $-1$~mm. All methods recovered the main circuit structure, including the large device regions and diagonal interconnects (Fig.~\ref{fig:chip_reconstructions}). On the shared magnitude scale, GM-MAGPIE gave the clearest visual separation of these structures, while the phase images showed method-dependent contrast and background texture. The four recovered probes also share the same central ring pattern (Fig.~\ref{fig:chip_probe_reconstructions}). The geometric-mean variants give a more localized magnitude background, while differences in the weak-field phase remain visible near the probe boundary. As in the test-pattern experiment, these observations describe visual structural delineation rather than quantitative accuracy.

\begin{figure}[tbp]
    \centering
    \includegraphics[width=\linewidth,height=0.82\textheight,keepaspectratio]{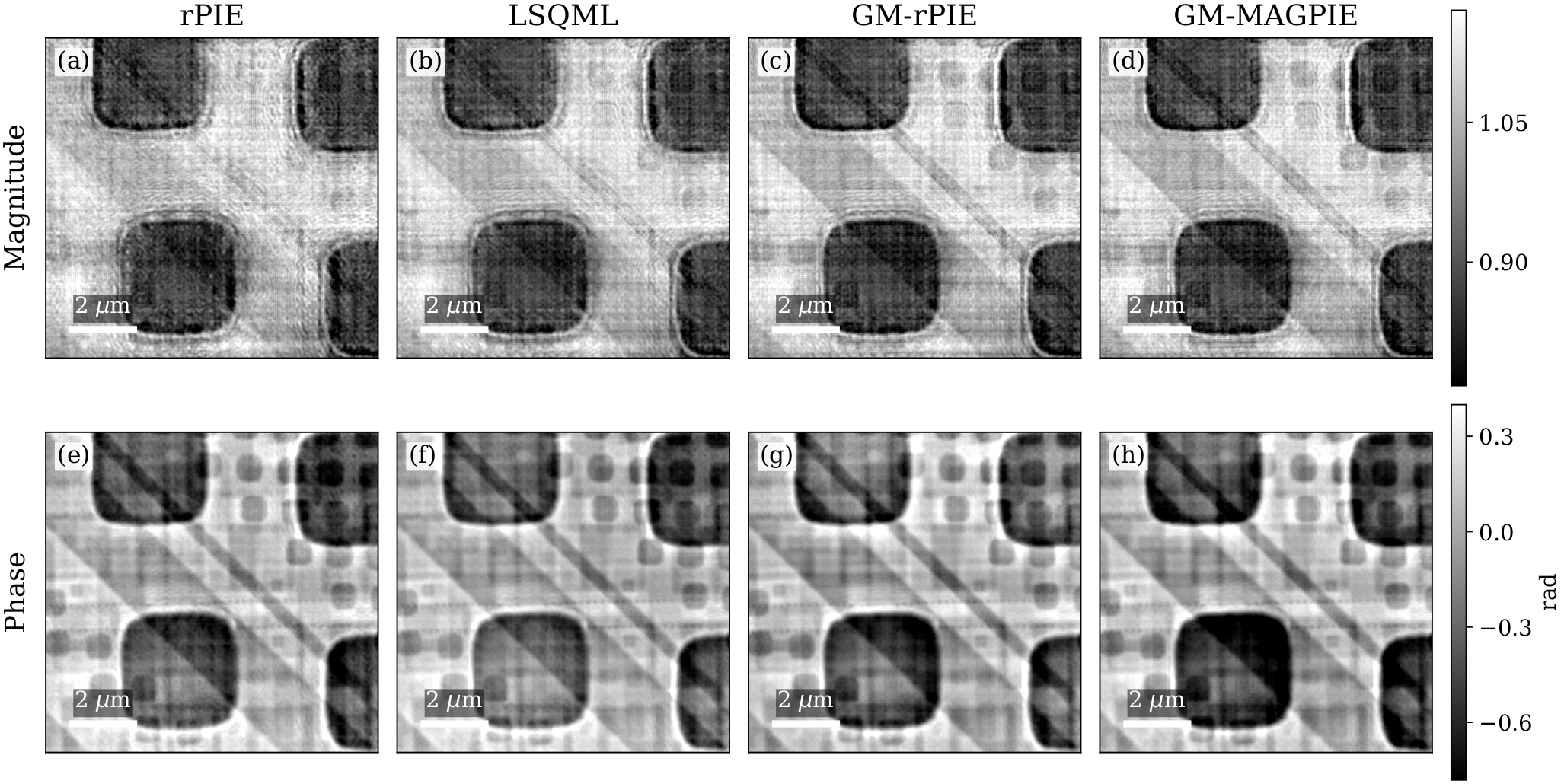}
    \caption{Integrated-circuit object reconstructions from all 812 patterns after 100 epochs. The rows show magnitude and phase, while the columns show rPIE, LSQML, GM-rPIE, and GM-MAGPIE. Each row shares display limits across methods. Each panel includes a $2~\mu\mathrm m$ scale bar.}
    \label{fig:chip_reconstructions}
\end{figure}

\begin{figure}[htbp]
    \centering
    \includegraphics[width=\linewidth]{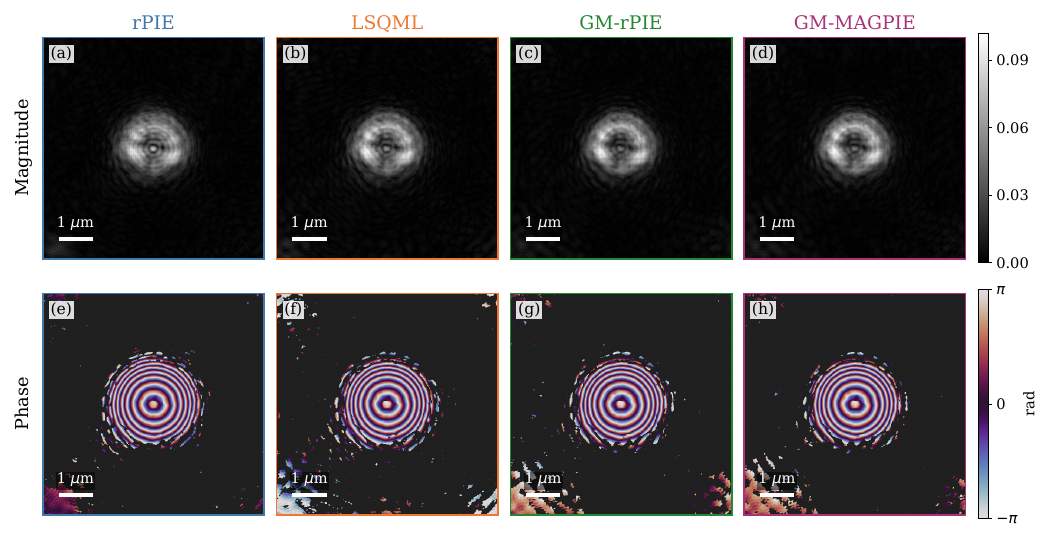}
    \caption{Integrated-circuit probe reconstructions after 100 epochs, corresponding to Fig.~\ref{fig:chip_reconstructions}. The rows show magnitude and phase, while the columns show rPIE, LSQML, GM-rPIE, and GM-MAGPIE. Each row shares display limits across methods. Each panel includes a $1~\mu\mathrm m$ scale bar. Phase uses cyclic colors, with weak probe phases masked.}
    \label{fig:chip_probe_reconstructions}
\end{figure}

\FloatBarrier

\section{Conclusion}
We introduced a coupled stochastic multigrid framework for blind ptychography that places minibatch processing and multilevel object correction within a common optimization construction. Starting from the same frozen surrogate, the framework coordinates object and probe updates locally, transfers object information across spatial scales, and synthesizes scan-dependent updates into shared fields. The analysis provides local residual control for the fine-grid geometric-mean update, characterizes how multilevel object corrections enter the coordinated residual, and motivates the intensity-weighted synthesis of each minibatch. Using only the fine grid gives GM-rPIE, whereas incorporating the object-grid hierarchy gives GM-MAGPIE.

At the prescribed budgets, GM-MAGPIE achieved the smallest final object error among the four methods in all nine tested synthetic overlap--noise configurations, while GM-rPIE or GM-MAGPIE achieved the smallest probe error. Its object-error advantage over GM-rPIE also persisted across the three tested minibatch sizes. Reconstructions of test-pattern and integrated-circuit data from the Advanced Photon Source demonstrate applicability to measured data, but the absence of experimental ground truth precludes an accuracy ranking.

Together, the analysis and experiments show how stochastic data processing and multigrid correction can be coupled without separating the object and probe reconstruction problems. The present guarantees are local rather than a global convergence theory, and the reported comparisons do not establish statistical robustness or computational speedup. Multi-seed studies, parameter-sensitivity analysis, computational scaling measurements, and quantitative experimental validation are therefore important next steps toward large-scale coherent-imaging applications.

\section*{Funding}
This material was based upon work supported by the U.S. Department of Energy, Office of Science, Offices of Advanced Scientific Computing Research and Basic Energy Sciences under contract DE-AC02-06CH11357.

\section*{Acknowledgments}
Computations used resources and services of the Partnership for an Advanced Computing Environment (PACE; RRID:SCR\_027619) at the Georgia Institute of Technology, Atlanta, Georgia, USA.

The authors used OpenAI ChatGPT with Codex (GPT-5, OpenAI; accessed August 6, 2026) as an interactive aid for manuscript formatting, grammar checking, organization, and mathematical theory development and verification. Codex was also used to prepare plotting code and assemble the numerical section from archived experimental outputs. The tool did not generate or alter research data or numerical reconstruction outputs.  The authors remain responsible for independently checking the derivations and for all mathematical statements and manuscript content.

\section*{Disclosures}
The authors declare no conflicts of interest.

\section*{Data Availability}
The code, experiment notebooks, and results reported in this paper are available at \url{https://github.com/borongzhang/blind_magpie_ptychography}. Raw experimental data may be obtained from the corresponding author upon reasonable request.

\appendix
\section{Surrogate and Reconstruction Details}
\label{app:reconstruction_details}

\subsection{Frozen surrogate and rPIE regularization}
\label{supp:derivation}

For the real-valued objectives considered here, we define the complex gradient as twice the Wirtinger derivative with respect to the conjugate field. In particular, $\nabla_{\bm P_k}\Phi_k\defeq2\partial\Phi_k/\partial\overline{\bm P_k}$, with the same convention for the object gradient and the surrogate~\cite{CRcalculus}. The following proposition establishes the surrogate's majorization and gradient-matching properties.

\begin{proposition}[Majorization and gradient matching]
\label{prop:surrogate_majorization}
For every object--probe pair $(\bm P,\bm O)$,
\[
    \Phi(\bm P,\bm O)\leq\widetilde\Phi(\bm P,\bm O;j),
    \qquad
    \Phi(\bm P^{(j)},\bm O^{(j)})
    =\widetilde\Phi(\bm P^{(j)},\bm O^{(j)};j).
\]
At each scan position $k$ for which
$\mathcal F(\bm P_k^{(j)}\odot\bm O_k^{(j)})$
has no zero entries, the gradients agree:
\begin{align*}
    \nabla_{\bm P_k}\Phi_k
        (\bm P_k^{(j)},\bm O_k^{(j)})
    &=\nabla_{\bm P_k}\widetilde\Phi_k
        (\bm P_k^{(j)},\bm O_k^{(j)};j)
      =-\overline{\bm O_k^{(j)}}\odot\bm r_k^{(j)},\\
    \nabla_{\bm O_k}\Phi_k
        (\bm P_k^{(j)},\bm O_k^{(j)})
    &=\nabla_{\bm O_k}\widetilde\Phi_k
        (\bm P_k^{(j)},\bm O_k^{(j)};j)
      =-\overline{\bm P_k^{(j)}}\odot\bm r_k^{(j)}.
\end{align*}
\end{proposition}

The rPIE regularization reflects confidence in the object estimate obtained from the revised exit wave. Where the probe is dim, division by the probe is poorly conditioned and can amplify noise, so a larger penalty favors retaining the current object estimate. Stronger illumination supports more reliable corrections and a smaller penalty. Reducing $\alpha_O$ relaxes this penalty and increases the contribution of the revised estimate at partially illuminated entries. The probe regularizer has the same interpretation with the object and probe interchanged and $\alpha_P$ controlling its strength.

\subsection{Multigrid object construction}
\label{supp:multigrid}

The multigrid object update follows MAGPIE~\cite{magpie}; we describe two levels, with the same construction repeated for additional levels.

For even $m$, let $I_h^H$ average disjoint $2\times2$ blocks and $I_H^h=4(I_h^H)^\top$ prolong coarse values. For the current scan update, with the epoch counter $j$ held fixed, set $\bm P_h=\bm P_k^{(j)}$, $\bm O_h=\bm O_k^{(j)}$, $\bm R_h=\bm R_k^{(j)}$, and $\bm u_h=\bm u_{\bm O}$. The fine and coarse frozen objectives are
\begin{equation}
    \widetilde\Phi_g(\bm U)
    =\tfrac12\|\bm R_g-\bm P_g\odot\bm U\|_2^2,
    \qquad g\in\{h,H\}.
    \label{eq:magpie-fine}
\end{equation}
The transfer weights account for the probe intensity in each coarse bin:
\begin{equation}
    \bm P_H=I_h^H\bm P_h,\qquad
    \bm W_{\bm O}
    =\frac{|\bm P_h|^2}{I_H^h I_h^H|\bm P_h|^2},\qquad
    \bm W_{\mathcal R}
    =\frac{(I_H^h\bm P_H)\odot\overline{\bm P_h}}
          {I_H^h I_h^H|\bm P_h|^2},\qquad
    \bm W_{\bm u}^H
    =\frac{|\bm P_H|^2}{I_h^H|\bm P_h|^2}.
    \label{eqn:def_weights}
\end{equation}
For bins with zero probe power, set $\bm W_{\bm O}=\bm W_{\mathcal R}=0$ and $\bm W_{\bm u}^H=1$. Restrict the object, revised exit wave, and regularizer by
\begin{equation}
    \bm O_H=I_h^H(\bm W_{\bm O}\odot\bm O_h),\qquad
    \bm R_H=I_h^H(\bm W_{\mathcal R}\odot\bm R_h),\qquad
    \bm u_H=\bm W_{\bm u}^H\odot I_h^H\bm u_h.
    \label{eq:downsampled_terms}
\end{equation}
The coarse regularizer is therefore inherited from the fine grid and requires no separately tuned parameter.

At either level, the regularized update from an input $\bm V$ is
\begin{equation}
\begin{aligned}
    \mathcal S_g(\bm V)
    &\defeq\operatorname*{argmin}_{\bm U}
    \left\{\widetilde\Phi_g(\bm U)
    +\tfrac12\langle\bm u_g,|\bm U-\bm V|^2\rangle\right\}\\
    &=\bm V+
    \frac{\overline{\bm P_g}}{|\bm P_g|^2+\bm u_g}
    \odot(\bm R_g-\bm P_g\odot\bm V).
\end{aligned}
    \label{eq:magpie_level_smoother}
\end{equation}
A zero denominator gives a zero correction. At each level, the target $\bm R_g$ remains fixed while constructing the current scan's multigrid update. The two-level update first updates the coarse object, prolongs the resulting correction to the fine grid, and then applies a fine-grid regularized update:
\begin{equation}
    \widehat{\bm O}_H=\mathcal S_H(\bm O_H),\qquad
    \bm V=\bm O_h+I_H^h(\widehat{\bm O}_H-\bm O_H),\qquad
    \widehat{\bm O}_{k,\mathrm{MG}}^{(j)}=\mathcal S_h(\bm V).
    \label{eq:coarse-prox}
\end{equation}
Additional levels repeat the coarse step recursively. 

\subsection{Discrete probe shifts and synthesis}
\label{supp:probe_interpolation}

To apply $S_k$, we pad the probe with a one-pixel border of replicated boundary values, apply the subpixel shift using the Fourier shift theorem on the padded grid, and crop the result to the original probe dimensions.

In the implementation, we form a weighted average of the adjoint-shifted local probe
corrections and add the result to the current probe. The weights are
computed by interpolating the complex object proposals with the opposite
displacement before taking their squared magnitudes. These steps agree
with Eq.~\eqref{eq:minibatch_shifted_probe_synthesis} for ideal
translations; padded discrete interpolation makes this an approximate
description. At zero total weight, the approximate least-squares problem
leaves the probe unconstrained. Where the implementation's weights have
zero sum, it adds the ordinary mean of the adjoint-shifted local corrections.

\section{Geometric-Mean and Multigrid Analysis}
\label{app:geometric_mean_descent}

This appendix gives the contraction and perturbation proofs and the transfer and coarse-correction bounds used in the multigrid construction. Each argument freezes the working pair for the current scan update; the epoch counter $j$ remains fixed throughout the argument. Scalar notation is introduced explicitly in each entrywise proof.

\subsection{Phase convention for the complex geometric mean}

For the principal argument in $(-\pi,\pi]$, the principal square root is defined by
\begin{equation}
    \operatorname{sqrt}_{\mathrm{pr}}(0)\defeq0,\qquad
    \operatorname{sqrt}_{\mathrm{pr}}(\rho e^{\mathrm i\theta})
    \defeq\sqrt\rho\,e^{\mathrm i\theta/2},
    \quad \rho>0,\quad\theta\in(-\pi,\pi].
    \label{eq:principal_square_root}
\end{equation}
In particular, a negative real number has a positive imaginary principal square root. This choice fixes the boundary of the phase interval used below.

For $a\ne0$, the geometric mean in Definition~\ref{def:anchored_geometric_mean} has the equivalent form
\begin{equation}
    \operatorname{gm}(a,b)
    =\frac{a}{|a|}\operatorname{sqrt}_{\mathrm{pr}}(b\overline a),
    \qquad\operatorname{gm}(a,b)^2=ab.
    \label{eq:anchored_geometric_mean_stable_form}
\end{equation}
To verify the first identity, substitute $b/a=b\overline a/|a|^2$ into the definition:
\[
\begin{aligned}
    \operatorname{gm}(a,b)
    &=a\operatorname{sqrt}_{\mathrm{pr}}
      \!\left(\frac{b\overline a}{|a|^2}\right)\\
    &=\frac{a}{|a|}\operatorname{sqrt}_{\mathrm{pr}}(b\overline a).
\end{aligned}
\]
The second equality holds because division by the positive real number $|a|^2$ leaves the argument unchanged and divides the principal square root by $|a|$. Squaring the original definition gives $a^2(b/a)=ab$. When $a=0$, the definition gives $\operatorname{gm}(0,b)=0$, so the square identity remains valid.

For $a\ne0$, the relative mean $\operatorname{gm}(a,b)/a$ is the principal square root of $b/a$. Whenever the mean is nonzero, its relative phase therefore lies in $(-\pi/2,\pi/2]$. This phase restriction is what allows the two block updates to be combined without losing the exit-wave residual bound.

\subsection{Proof of entrywise exit-wave residual contraction}

\begin{proof}[Proof of Proposition~\ref{prop:gm_rpie_local_descent}] Freeze the working pair, revised exit wave, and regularizers for the current scan update at position $k$, with the epoch counter $j$ held fixed. It suffices to work entrywise. Fix an entry index $s$, and write
\[
    p=P_{k,s}^{(j)},\qquad
    o=O_{k,s}^{(j)},\qquad
    r=R_{k,s}^{(j)},\qquad
    u_O=[\bm u_{\bm O}]_s,\qquad
    u_P=[\bm u_{\bm P}]_s.
\]
Thus $p$, $o$, and $r$ are entries of the probe, object, and revised exit wave, respectively; the current product is $po$ and the residual entry is $r-po$. Let $\widehat p,\widehat o$ denote the separate block updates and $\widetilde p,\widetilde o$ their coordinated values. The rPIE formulas in Eqs.~\eqref{eq:rpie_object_minimizer}--\eqref{eq:rpie_probe_minimizer} give
\begin{equation}
    \widehat o=o+\frac{\overline p}{|p|^2+u_O}(r-po),\qquad
    \widehat p=p+\frac{\overline o}{|o|^2+u_P}(r-po).
    \label{eq:appendix_scalar_proximal_updates}
\end{equation}
Define the scalar contraction factors
\[
    \kappa_O=\frac{u_O}{|p|^2+u_O},\qquad
    \kappa_P=\frac{u_P}{|o|^2+u_P}.
\]
These are the corresponding entries of Eq.~\eqref{eq:gm_contraction_factors} and lie in $[0,1]$. Multiplying the object and probe updates by the opposite current field gives
\begin{equation}
    p\widehat o=r-\kappa_O(r-po),\qquad
    \widehat p\,o=r-\kappa_P(r-po).
    \label{eq:appendix_collinear_products}
\end{equation}
Both products therefore lie on the line segment joining $r$ to $po$. Write $w=\widetilde p\,\widetilde o$ for the combined product.

First suppose that $po=0$. Then $p=0$ or $o=0$, so the corresponding geometric mean is zero and $w=0$. Positive denominators imply $\kappa_O=1$ when $p=0$ and $\kappa_P=1$ when $o=0$. Consequently,
\[
    |r-w|=|r-po|
    =\max\{\kappa_O,\kappa_P\}|r-po|.
\]
This proves the bound at every zero current product.

Now suppose that $po\ne0$. Define the relative means
\[
    A_P=\frac{\widetilde p}{p},\qquad
    A_O=\frac{\widetilde o}{o}.
\]
By Definition~\ref{def:anchored_geometric_mean},
\[
    A_P=\operatorname{sqrt}_{\mathrm{pr}}\!\left(
          \frac{\widehat p}{p}\right),\qquad
    A_O=\operatorname{sqrt}_{\mathrm{pr}}\!\left(
          \frac{\widehat o}{o}\right).
\]
Using Eq.~\eqref{eq:appendix_collinear_products}, their squares satisfy
\begin{equation}
    A_P^2=\kappa_P+(1-\kappa_P)\frac{r}{po},\qquad
    A_O^2=\kappa_O+(1-\kappa_O)\frac{r}{po}.
    \label{eq:appendix_relative_squares}
\end{equation}
The two radicands lie on the same line segment joining $1$ to $r/(po)$, so their imaginary parts have the same sign. Their principal square roots therefore lie in a common closed quadrant. This also holds when the radicands are real: their principal roots lie on the nonnegative real or positive imaginary axis. Zero roots cause no exception. Hence
\begin{equation}
    \operatorname{Re}(A_O\overline{A_P})\geq0.
    \label{eq:appendix_common_quadrant}
\end{equation}
Since
\[
    w=poA_PA_O,\qquad
    p\widehat o=poA_O^2,\qquad
    \widehat p\,o=poA_P^2,
\]
we obtain
\begin{equation}
\begin{aligned}
    \operatorname{Re}\!\left[
       (w-p\widehat o)\overline{(w-\widehat p\,o)}\right]
    &=-|po|^2|A_P-A_O|^2
       \operatorname{Re}(A_O\overline{A_P})\leq0.
\end{aligned}
    \label{eq:appendix_thales_criterion}
\end{equation}
This inequality places $w$ in the closed disk whose diameter joins $p\widehat o$ and $\widehat p\,o$. By Eq.~\eqref{eq:appendix_collinear_products}, its center and radius are, respectively,
\[
    r-\frac{\kappa_O+\kappa_P}{2}(r-po),\qquad
    \frac{|\kappa_O-\kappa_P|}{2}|r-po|.
\]
The triangle inequality now gives
\[
\begin{aligned}
    |r-w|
    &\leq\frac{\kappa_O+\kappa_P}{2}|r-po|
       +\frac{|\kappa_O-\kappa_P|}{2}|r-po|\\
    &=\max\{\kappa_O,\kappa_P\}|r-po|.
\end{aligned}
\]
Because $r-po=r_{k,s}^{(j)}$, this is precisely the scalar form of Eq.~\eqref{eq:entrywise_gm_contraction}. Moreover,
\[
    \max\{|r-p\widehat o|^2,|r-\widehat p\,o|^2\}
    =\max\{\kappa_O^2,\kappa_P^2\}|r-po|^2.
\]
Squaring and summing over $s$ proves Eq.~\eqref{eq:vector_gm_contraction}. If $p,o\ne0$, then both contraction factors are strictly less than one. Thus any such entry with $r-po\ne0$ makes the frozen surrogate inequality for the current scan strict, as stated in the proposition.
\end{proof}

\subsection{Transfer bounds for the multigrid object update}
\label{app:magpie_transfer}

For the current scan update, hold the fine-grid probe, target, and regularizer fixed while constructing the coarse level. This two-level construction inherits the transfer bounds of MAGPIE~\cite{magpie}:
\begin{equation}
    \|\bm W_{\bm O}\|_\infty,
    \ \|\bm W_{\mathcal R}\|_\infty\leq4,\qquad
    \|\bm W_{\bm u}^H\|_\infty\leq1.
    \label{eq:magpie_transfer_weight_bounds}
\end{equation}
The first two follow by comparing each value with the average power in its four-entry bin; the last follows from $|I_h^H\bm P_h|^2\leq I_h^H|\bm P_h|^2$. The restricted residual and gradient obey
\begin{equation}
\begin{aligned}
    \bm R_H-\bm P_H\odot\bm O_H
    &=I_h^H\!\left(
    \bm W_{\mathcal R}\odot(\bm R_h-\bm P_h\odot\bm O_h)
    \right),\\
    \nabla\widetilde\Phi_H(\bm O_H)
    &=\bm W_{\bm u}^H\odot
      I_h^H\nabla\widetilde\Phi_h(\bm O_h).
\end{aligned}
    \label{eq:magpie_transfer_residual_gradient}
\end{equation}
Because $\|I_h^H\|_2=1/2$, these identities give
\begin{equation}
\begin{aligned}
    \widetilde\Phi_H(\bm O_H)
    &\leq\tfrac14\|\bm W_{\mathcal R}\|_\infty^2
        \widetilde\Phi_h(\bm O_h),\\
    \|\nabla\widetilde\Phi_H(\bm O_H)\|_2
    &\leq\tfrac12\|\bm W_{\bm u}^H\|_\infty
        \|\nabla\widetilde\Phi_h(\bm O_h)\|_2.
\end{aligned}
    \label{eq:magpie_transfer_objective_gradient_bounds}
\end{equation}
The prolonged coarse correction also has a nonpositive directional derivative for the fine frozen objective:
\begin{equation}
\begin{aligned}
    &\operatorname{Re}\left\langle
    \nabla\widetilde\Phi_h(\bm O_h),
    I_H^h(\widehat{\bm O}_H-\bm O_H)\right\rangle\\
    &\quad=-4\left\langle
    \frac{\bm W_{\bm u}^H}{|\bm P_H|^2+\bm u_H},
    |I_h^H\nabla\widetilde\Phi_h(\bm O_h)|^2
    \right\rangle\leq0.
\end{aligned}
    \label{eq:magpie_coarse_direction}
\end{equation}
Terms with zero denominator are assigned zero. A nonpositive directional derivative does not by itself guarantee that the full prolonged step decreases the objective.

The inherited regularizer also controls the size of the coarse correction. The comparison below concerns the largest correction produced by a bounded fine-grid gradient.

\begin{proposition}[Coarse-correction bound]
\label{prop:magpie_coarse_correction_bound}
Suppose that $|\bm P_h|^2+\bm u_h$ is entrywise positive and that $\bm u_H$ is chosen as in Eq.~\eqref{eq:downsampled_terms}. As functions of a fine-grid gradient $\bm g$, the fine and coarse corrections are
\[
    \bm c_h(\bm g)
    =-\frac{\bm g}{|\bm P_h|^2+\bm u_h},\qquad
    \bm c_H(\bm g)
    =-\frac{\bm W_{\bm u}^H}{|\bm P_H|^2+\bm u_H}
       \odot I_h^H\bm g,
\]
where a zero denominator in the coarse correction gives zero. Then, with the maxima taken entrywise on the coarse grid,
\begin{equation}
    \max_{\|\bm g\|_\infty\leq1}|\bm c_H(\bm g)|
    \leq
    \max_{\|\bm g\|_\infty\leq1}
       |I_h^H\bm c_h(\bm g)|.
    \label{eq:magpie_coarse_correction_bound}
\end{equation}
Thus the largest coarse correction produced by a bounded fine-grid gradient is controlled by the largest restricted fine-grid correction under the same gradient bound.
\end{proposition}
\begin{proof}
The inherited regularizer gives
\[
    |\bm P_H|^2+\bm u_H
    =\bm W_{\bm u}^H\odot
       I_h^H\bigl(|\bm P_h|^2+\bm u_h\bigr).
\]
Consequently, on a bin where $\bm W_{\bm u}^H$ is positive,
\[
    \bm c_H(\bm g)
    =-\frac{I_h^H\bm g}
            {I_h^H(|\bm P_h|^2+\bm u_h)}.
\]
Because $I_h^H$ averages four entries, $|I_h^H\bm g|\leq1$ when $\|\bm g\|_\infty\leq1$, with equality attained by $\bm g=\bm1$. The coarse correction is zero on the remaining bins. It follows that
\[
    \max_{\|\bm g\|_\infty\leq1}|\bm c_H(\bm g)|
    \leq\frac{1}{I_h^H(|\bm P_h|^2+\bm u_h)}
    \leq I_h^H\!\left(\frac{1}{|\bm P_h|^2+\bm u_h}\right),
\]
where the second inequality is Jensen's inequality for the convex function $t\mapsto1/t$ on $(0,\infty)$. For the restricted fine correction, the triangle inequality gives
\[
    |I_h^H\bm c_h(\bm g)|
    \leq I_h^H\!\left(\frac{1}{|\bm P_h|^2+\bm u_h}\right)
    \qquad (\|\bm g\|_\infty\leq1),
\]
and $\bm g=\bm1$ attains equality. Combining these observations proves the result.
\end{proof}

\subsection{Perturbation by the multigrid object update}
\label{app:magpie_perturbation}

\begin{proof}[Proof of
Proposition~\ref{prop:magpie_geometric_mean_perturbation}] Freeze the working pair and revised exit wave for the current scan update at position $k$, with the epoch counter $j$ held fixed. Fix an entry index $s$ satisfying the assumptions, and use the same scalar notation as in the contraction proof:
\[
    p=P_{k,s}^{(j)},\qquad
    o=O_{k,s}^{(j)},\qquad
    r=R_{k,s}^{(j)}.
\]
Here $po\ne0$ and $r\ne0$. Let $\widehat p$ be the rPIE probe-update entry and $\widehat o_{\mathrm{MG}}$ the multigrid object-update entry. Write
\[
\begin{aligned}
    \widetilde p&=\operatorname{gm}(p,\widehat p),\qquad
    \widetilde o=\operatorname{gm}(o,\widehat o_{\mathrm{MG}}),\qquad
    w=\widetilde p\,\widetilde o,\\
    e_O&=r-p\widehat o_{\mathrm{MG}},\qquad
    e_P=r-\widehat p\,o.
\end{aligned}
\]
The scalars $e_O$ and $e_P$ are the corresponding entries of $\bm e_{\bm O,k}^{(j)}$ and $\bm e_{\bm P,k}^{(j)}$ in the proposition.

The square identity for the geometric mean in Definition~\ref{def:anchored_geometric_mean} gives
\[
    w^2=(p\widehat o_{\mathrm{MG}})(\widehat p\,o)
       =(r-e_O)(r-e_P).
\]
Completing the square shows how the average residual appears:
\[
    w^2=\left(r-\frac{e_O+e_P}{2}\right)^2
         -\frac{(e_O-e_P)^2}{4}.
\]
Factoring the difference of squares therefore yields
\[
    \left(w-r+\frac{e_O+e_P}{2}\right)
    \left(w+r-\frac{e_O+e_P}{2}\right)
    =-\frac{(e_O-e_P)^2}{4}.
\]
By the proposition's assumption, $w\to r$ as $e_O,e_P\to0$, with $r\ne0$ held fixed. The second factor tends to $2r\ne0$, so division is valid in a sufficiently small neighborhood:
\[
    r-w=\frac{e_O+e_P}{2}
       +\frac{(e_O-e_P)^2}
        {4\left(w+r-\frac{e_O+e_P}{2}\right)}.
\]
The second term is quadratic in the separate-update residuals, because its denominator stays bounded away from zero. More precisely,
\[
    r-w=\frac{e_O+e_P}{2}
       +\mathcal O\!\left(\frac{|e_O-e_P|^2}{|r|}\right).
\]
Finally, $|e_O-e_P|\leq|e_O|+|e_P|$ gives the remainder in Eq.~\eqref{eq:local_multigrid_gm_expansion}. The argument holds at each admissible entry, proving the vector statement.
\end{proof}

\begingroup
\small
\bibliographystyle{unsrt}
\bibliography{references}
\endgroup
\end{document}